\documentclass[conf]{new-aiaa}
\usepackage[utf8]{inputenc}

\usepackage{graphicx}
\usepackage{amsmath}
\usepackage{amsmath,bm}
\usepackage{amsfonts}
\usepackage{subcaption}
\usepackage{xcolor}
\usepackage[T1]{fontenc}
\usepackage{rotating} 
\usepackage{array}

\usepackage{amsthm}
\usepackage{booktabs}
\usepackage[version=4]{mhchem}
\usepackage{siunitx}
\usepackage{longtable,tabularx}
\usepackage{algorithm}
\usepackage{algpseudocode}
\usepackage[acronyms, shortcuts]{glossaries}
\usepackage{hyperref}

\hypersetup{
    colorlinks=true,    
    linkcolor=black,     
    citecolor=black,     
    filecolor=blue,     
    urlcolor=blue       
}

\newacronym[longplural={quadratic programs}]{qp}{QP}{quadratic program}
\newacronym[longplural={second-order cone program}]{socp}{SOCP}{second-order cone program}
\newacronym[longplural={interior point methods}]{ipm}{IPM}{interior point method}
\newacronym[longplural={sequential quadratic programs}]{sqp}{SQP}{sequential quadratic program}
\newacronym[longplural={sequential convex programs}]{scp}{SCP}{sequential convex program}
\newacronym[longplural={nonlinear programs}]{nlp}{NLP}{nonlinear program}
\newacronym[longplural={mixed-integer convex programs}]{micp}{MICP}{mixed-integer convex program}

\newacronym{licq}{LICQ}{Linear Independence Constraint Qualification}
\newacronym{kkt}{KKT}{Karush–Kuhn–Tucker}
\newacronym{pmp}{PMP}{Pontryagin's Maximum Principle}

\newacronym{scvx}{SCvx}{Successive Convexification}
\newacronym{lcvx}{LCvx}{Lossless Convexification}
\newacronym{tscvx}{T-SCvx}{Transformer-based Successive Convexification}
\newacronym{upg}{UPG}{Universal Powered Guidance}
\newacronym{pdg}{PDG}{Powered Descent Guidance}
\newacronym{tpdg}{T-PDG}{Transformer-based Powered Descent Guidance}

\newacronym[longplural={degrees-of-freedom}]{dof}{DoF}{degree-of-freedom}
\newacronym{pdf}{PDF}{probability density function}

\newacronym[longplural={Denoising Diffusion Probabilistic Models}]{ddpm}{DDPM}{Denoising Diffusion Probabilistic Model}
\newacronym[longplural={Deep Neural Networks}]{dnn}{DNN}{Deep Neural Network}

\newacronym{ido}{IDO}{Iterative Diffusion Optimization}
\newacronym{socm}{SOCM}{stochastic optimal control matching}

\newacronym{ood}{OOD}{out-of-distribution}

\newtheorem{definition}{Definition}
\newenvironment{claim}[1]{\par\noindent\underline{Claim:}\space#1}{}

\newtheorem{theorem}{Theorem}[section]

\newcommand{\ie}{{i.e.}}

\title{Tight Constraint Prediction of Six-Degree-of-Freedom Transformer-based Powered Descent Guidance}

\author{Julia Briden\footnote{Advanced Mission Design and GN\&C Engineer, Amentum, NASA Johnson Space Center, 2101 E NASA Pky, Houston, TX 77058 USA, and AIAA Member.}
}
\affil{2101 E NASA Pky, Houston, TX, 77058 USA}

\author{Trey Gurga \footnote{Undergraduate Student, Department of Aeronautics and Astronautics, 77 Massachusetts Avenue, Cambridge, Massachusetts, 02139 USA, and AIAA Student Member.}}

\affil{Department of Aeronautics and Astronautics, Massachusetts Institute of Technology, 77 Massachusetts Avenue, Cambridge, Massachusetts, 02139 USA}

\author{Breanna Johnson \footnote{Aerospace Engineer, Flight Mechanics and Trajectory Design Branch, EG5, NASA Johnson Space Center, 2101 E NASA Pky, Houston, TX, Senior Member AIAA.}}

\affil{2101 E NASA Pky, Houston, TX, 77058 USA}

\author{Abhishek Cauligi \footnote{Robotics Technologist, Jet Propulsion Laboratory, California Institute of Technology, Pasadena, CA 91109, USA.}}

\affil{NASA Jet Propulsion Laboratory, California Institute of Technology, 4800 Oak Grove Dr, Pasadena, CA 91109 USA}

\author{Richard Linares \footnote{Rockwell International Career Development Professor and Associate Professor, Department of Aeronautics and Astronautics, 125 Massachusetts Avenue. Senior Member AIAA.}}

\affil{Department of Aeronautics and Astronautics, Massachusetts Institute of Technology, 77 Massachusetts Avenue, Cambridge, Massachusetts, 02139 USA}

\begin{document}

\maketitle

\begin{abstract}
This work introduces Transformer-based Successive Convexification (T-SCvx), an extension of Transformer-based Powered Descent Guidance (T-PDG), generalizable for efficient six-degree-of-freedom (DoF) fuel-optimal powered descent trajectory generation. Our approach significantly enhances the sample efficiency and solution quality for nonconvex-powered descent guidance by employing a rotation invariant transformation of the sampled dataset.
T-PDG was previously applied to the 3-DoF minimum fuel powered descent guidance problem, improving solution times by up to an order of magnitude compared to lossless convexification (LCvx).
By learning to predict the set of tight or active constraints at the optimal control problem's solution, Transformer-based Successive Convexification (T-SCvx) creates the minimal reduced-size problem initialized with only the tight constraints, then uses the solution of this reduced problem to warm-start the direct optimization solver.
6-DoF powered descent guidance is known to be challenging to solve quickly and reliably due to the nonlinear and non-convex nature of the problem, the discretization scheme heavily influencing solution validity, and reference trajectory initialization determining algorithm convergence or divergence.
Our contributions in this work address these challenges by extending T-PDG to learn the set of tight constraints for the successive convexification (SCvx) formulation of the 6-DoF powered descent guidance problem.
In addition to reducing the problem size, feasible and locally optimal reference trajectories are also learned to facilitate convergence from the initial guess.
T-SCvx enables onboard computation of real-time guidance trajectories, demonstrated by a 6-DoF Mars powered landing application problem.
\end{abstract}

\section{Nomenclature}

\textbf{Variables}{\renewcommand\arraystretch{1.0}
\noindent\begin{longtable*}{@{}l @{\quad=\quad} l@{}}
$\alpha_{\dot m}$ & fuel consumption rate \\
$\beta_{\dot m}$ & vacuum specific impulse \\
$\delta_{\max}$ & maximum gimbal angle \\
$\nu$ & virtual control \\
$\bm{\Omega}_{\omega_\mathcal{B} (t)}$ & quaternion transform \\
$\omega$ & penalty coefficients \\
$\bm{\omega}_\mathcal{B}$ & angular velocity \\
$\phi$ & pitch angle \\
$\sigma^i$ & time-scaling factor \\
$\tau_i$ & set of tight constraints for sample $i$ \\
$\theta_i$ & neural network input parameters for sample $i$ \\
$\theta_{\max}$ & maximum pointing angle \\
$A$ & state dynamics matrix \\
$B$ & control dynamics matrix \\
$C$ & future control dynamics matrix \\
$C_{\mathcal{I} \leftarrow \mathcal{B}}$
& transformation matrix \\
$d_k$ & number of columns in the transformer weight matrix \\
$\bm{e}$ & inertial frame unit vector $(e_1, e_2, e_3)$, defined at the landing site, with $e_1$ pointing in the opposite direction of $\bm{g}_\mathcal{I}$ \\
$\bm{g}_\mathcal{I}$ & gravitational acceleration \\
$H_\gamma$ & glideslope matrix \\
$H_\theta$ & pointing angle matrix \\
$h_{\text{gs}}$  & glideslope constraint vector \\
$h = 1,...,H$  & transformer head \\
$I$ & identity \\
$\mathbf{J}_\mathcal{B}$ & moment of inertia \\
$K$ & final time step \\
$K_h$ & transformer key matrices \\
$k$ & iteration number \\
$M_\mathcal{B}$ & net propulsive and aerodynamic torque acting on the vehicle equal to $r_{T,\mathcal{B}} \times T_\mathcal{B} (t) + r_{\text{cp}, \mathcal{B}} \times A_\mathcal{B}(t)$ [N m] \\
$m$ & mass \\
$m_\text{wet}$ & wet mass \\
$N$  & number of discretization nodes \\
$O_h$ & transformer attention output \\
$Q_h$ & transformer key matrices \\
$P_\text{amb}$ & ambient atmospheric pressure [N/m$^2$] \\
$q$ & spacecraft quaternion \\
$\bm{r}_\mathcal{I}$ & position \\
$\mathbf{s}$ & state \\
$s_i$ & neural network strategy output for sample $i$ \\
$t_f$ & final time \\
$\mathbf{T}_\mathcal{B}$ & thrust \\
$T_{\max}$ & maximum thrust \\
$T_{\min}$ & minimum thrust \\
$V_h$ & transformer value matrices \\
$\bm{v}_\mathcal{I}$ & velocity \\
$z$ & optimization problem parameter
\end{longtable*}

\textbf{Functions}{\renewcommand\arraystretch{1.0}
\noindent\begin{longtable*}{@{}l @{\quad=\quad} l@{}}
$f: \mathbb{R}^n \rightarrow \mathbb{R}$ & cost function \\
$\boldsymbol{g}: \mathbb{R}^{n_x} \times \mathbb{R}^{n_u} \times \mathbb{R} \rightarrow  \mathbb{R}^{n_c}$ & vector of non-convex constraints \\
$\boldsymbol{h}: \mathbb{R}^{n_x} \times \mathbb{R}^{n_x} \times \mathbb{R} \rightarrow  \mathbb{R}^{n_b}$ & vector of convex constraints \\
\end{longtable*}}

\textbf{Notation}{\renewcommand\arraystretch{1.0}
\noindent\begin{longtable*}{@{}l @{\quad=\quad} l@{}}
$\otimes$ & quaternion multiplication \\
$\Omega$ & skew-symmetric matrix defined such that the quaternion kinematics hold~\cite{Shuster1993} \\
$\hat{e}$  & unit direction vector \\
$\theta$ & parametric problem parameters \\
$\mathcal{B}$ & problem parameters in the body-fixed reference frame \\
$\mathcal{I}_\text{eq}$ & set of non-convex equality constraints \\
$\mathcal{I}_\text{ineq}$ & set of non-convex inequality constraints \\
$\mathcal{J}_\text{ineq}$ & set of convex inequality constraints \\
$\mathcal{I}$ & problem parameters in the inertial reference frame \\
$q^*_{\mathcal{B} \leftarrow \mathcal{I}}$ & conjugate of $q_{\mathcal{B} \leftarrow \mathcal{I}}$ \\
$q_{\text{id}}$ & identity quaternion \\

\end{longtable*}}

\section{Introduction}
\lettrine{I}{ncreasingly} complex and high mass planetary missions require autonomous long-horizon trajectory generation to achieve dynamically feasible~\ac{pdg}.
Moreover, there is a pressing need to generate fuel-optimal trajectories in order to satisfy the mass and safety margin requirements necessary for enabling human-rated vehicles.
A key challenge in applying existing methods is that the radiation-hardened processors used in most aerospace applications struggle to satisfy the sub-second trajectory generation requirements necessary to enable onboard usage.
While current work in custom solver implementations for 6-\ac{dof}~\ac{pdg} has achieved sub-millisecond level runtimes, these only hold for short horizon problems of less than 50 discretization nodes~\cite{SzmukReynoldsEtAl2020}.
Recently, data-driven methods have emerged as a promising tool in reducing onboard computation times for numerous classes of algorithms~\cite{IzzoMartensEtAl2019,GuffantiGamelliEtAl2024} and researchers have introduced methods such as~\ac{tpdg} to bridge this gap in long-horizon trajectory generation \cite{BridenGurgaEtAl2024}.
This method utilizes previous runs of the 3-\ac{dof}~\ac{lcvx} implementation of the Mars-powered descent landing problem to train a non-linear mapping between problem parameters and the set of tight constraints that defines the optimal problem solution.
The reduced problem, defined by equality and identified tight constraints, is then solved and used as an initial guess for the full problem to ensure feasibility and constraint satisfaction. 
This work extends~\ac{tpdg} to 6-\ac{dof}, introducing~\ac{tscvx}, on a smaller rotation-invariant training and test set of trajectories generated using~\ac{scvx}.
Additionally, variable-horizon outputs, which include a trajectory and control guess, are predicted to facilitate fast and reliable convergence to a locally optimal solution.

By enabling the deployment of autonomous optimal guidance technologies for spacecraft powered descent landing trajectory generation, a wider range of dispersions and uncertainties can be recovered, improving the mission safety margins and enabling the exploration of increasingly challenging and scientifically rich landing sites.

Optimal solutions can be formulated either using indirect or direct methods.
Indirect methods, including~\ac{upg} and Propellant-Optimal~\ac{pdg}, solve for the necessary optimality conditions for the optimal control problem by solving a two-point boundary value problem corresponding to the state and costate dynamics and their boundary conditions \cite{LuForbesEtAl2012, Lu2018}.
For common classes of problems, application of Pontryagin's Maximum Principle~\cite{Kirk2012} yield analytical solutions, but these methods fall short for trajectory generation as they are applicable to a limited set of mission constraints and objective functions~\cite{Klumpp1974, NajsonMease2006, TopcuCasolivaEtAl2007, SostaricRea2005}.
Moreover, while both analytical and indirect methods are computationally efficient, significant simplifications of the dynamics and constraints are required for both problem formulations.
As such, bang-bang control strategies and linear gravity assumptions are often required and both state and control inequality constraints are not easily implementable in a root-finding framework~\cite{LuForbesEtAl2012, Lu2018}.

In contrast, direct methods, often formulated as~\acp{scp}~\cite{ByrdGouldEtAl2003, PalaciosGomezEtAl1982, HorstThoai1999, YilleRangarajan2001, LiuLu2014, LiuShenEtAl2015, MaoSzmukEtAl2016, WangGrant2017, MalyutaReynoldsEtAl2022} or~\acp{sqp} \cite{Han1997, Powell1978, PowellYuan1986, BoggsTolle1989, BoggsTolle1995, Fukushima1986, GillMurrayEtAl2005, BettsHuffman1993}, compute locally optimal solutions for a general class of non-convex constrained optimization problems.
A special case of direct methods is the~\ac{socp}, which results from the lossless convexification of nonconvex cost, dynamics, and control constraints, often applied to the 3-\ac{dof} powered descent guidance problem~\cite{AcikmeseBlackmore2011, HarrisAcikmese2014, BlackmoreAccikmeseEtAl2012, AccikmesePloenEtAl2005, AccikmesePloenEtAl20017, BlakmoreAccikmeseEtAl2010, AccikmeseCarsonEtAl2013, HarrisAccikmese2013, DueriZhangEtAl2014, DueriAcikmeseEtAl2017}.
In this case, the nonconvex problem formulation can be formulated as a convex~\ac{socp}, enabling the recovery of the globally optimal solution, including certificates of convergence and infeasibility via an~\ac{ipm} algorithm \cite{BoydVandenberghe2004,NocedalWright2006}.
An extension of this work based on dual-quaternions allowed for 6-\ac{dof} motion, but required piecewise-affine approximations of the nonlinear dynamics, degrading solution accuracy as the temporal resolution of the discretization decreases~\cite{LeeMesbahi2015, LeeMesbahi2016}.
This special case of problems also relies on a line search to recover the optimal final time due to the non-convexities introduced by allowing a free final time~\cite{MalyutaReynoldsEtAl2022}.

To generalize beyond the narrow class of problems that can be losslessly convexified,~\ac{scvx} can solve the general class of non-convex optimal control problems while forgoing the rigorous optimality and convergence guarantees provided by convex solvers~\cite{BoydVandenberghe2004}.
\ac{scvx} transforms the non-convex full problem into a sequence of convex subproblems, which can be solved successfully until convergence~\cite{LiuLu2014, LiuShenEtAl2015, MaoSzmukEtAl2016, WangGrant2017, MaoSzmukEtAl2016}.
\ac{scp}-based methods have been generalized beyond powered descent guidance to non-convex control-affine systems and incorporated indirect optimal control methods for accelerated convergence, as well as for~\acp{micp}~\cite{CauligiCulbertsonEtAl2020a, CauligiCulbertsonEtAl2022, CauligiChakrabartyEtAl2022}.
In this work, we focus on efficiently solving the free-final-time 6-\ac{dof}~\ac{scvx} problem formulation by predicting the reduced problem defined by only the tight constraints and a dynamically feasible initial guess via a transformer-based deep neural network.
Full constraint satisfaction is guaranteed by warm-starting the full problem with the reduced problem's solution.

This work aims to improve the computational efficiency of the 6-\ac{dof} powered descent guidance problem with free ignition time, including aerodynamic forces and conditional enforced constraints \cite{SzmukReynoldsEtAl2020}.
While this problem setup enables high-fidelity locally fuel-optimal trajectory generation via explicit inclusion of operational and mission constraints, directly solving the problem in real-time is challenging due to nonlinear dynamics and non-convex state and control constraints, which do not yield close-form solutions.
Furthermore, the convergence of this nonconvex problem depends heavily on both the discretization scheme and the development of a suitable initial guess trajectory to initialize the solver.
This work focuses on efficiently solving the free-final-time 6-\ac{dof}~\ac{scvx} problem formulation by predicting the reduced problem defined by only the tight constraints and a dynamically feasible initial guess via a transformer-based deep neural network.
Full constraint satisfaction is guaranteed by warm-starting the full problem with the reduced problem's solution.

\subsection{Contributions}
This work develops a 6-\ac{dof}~\ac{tscvx} with Mars landing, benchmarks against~\ac{scvx} algorithms, and lookup table-based methods.
\ac{tscvx} significantly improves the convergence and computational efficiency of successive convexification by 1) enabling efficient sub-problem generation and warm-starting using Transformer-based tight constraint prediction, 2) utilizing rotation-invariant data augmentation, lowering the number of optimization problem samples required for training to under 2,000 samples, and 3) ensuring feasibility by only warm starting the convex subproblems, computing and solving with the penalty cost for the full problem. 
Figure~\ref{alg:TSCvx} shows the~\ac{tscvx} inputs, $\theta_i$, and strategy, $s_i$, for a given test case, $i$.

\begin{figure}[hbt!]
    \centering
    \includegraphics[width=0.7\textwidth]{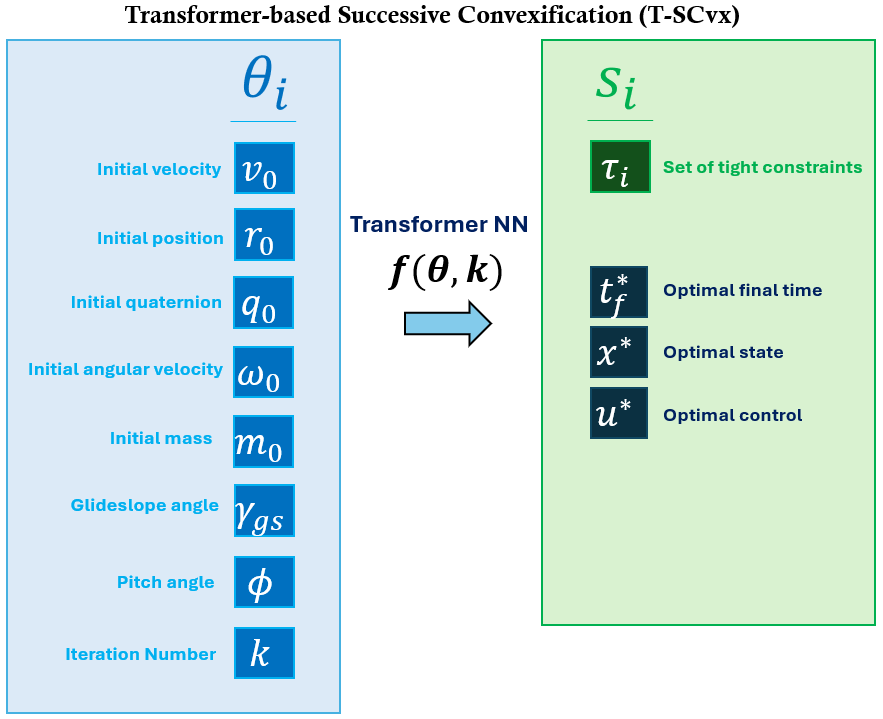}
    \caption{\ac{tscvx} inputs include the initial state and iteration number. The output of the constraint neural network, $\tau(\theta^{(i)})$, is the set of tight constraints, and the outputs of the initial guess neural network, $t_f^*(\theta^{(i)})$, $x_*(\theta^{(i)})$, and $u^*(\theta^{(i)})$, are the optimal final time, state, and control.}
    \label{fig:Tscvx}
\end{figure}

\subsection{Outline}

This paper is formulated as follows.
Section~\ref{sec:methods} defines the 6-\ac{dof} minimum-fuel powered descent guidance problem formulation, details the Successive Convexification Sequential Convex Programming solver, defines tight constraint prediction and the recovered optima achieved using~\ac{tpdg}, details the learning process, and defines the real-time test problem setup of~\ac{tscvx} for 6-\ac{dof} powered descent guidance.
The~\ac{tscvx} results and analysis are in Section \ref{sec: 6 DoF TPDG}.
Finally, areas of future work are described in Section~\ref{sec: Future Directions}, and conclusions are stated in Section~\ref{sec: Conclusion}.

\section{Methods}\label{sec:methods}
This section defines the generalized non-convex optimization problem and the associated notation, introduces the 6-\ac{dof} powered descent guidance problem and its~\ac{scvx} formulation, formalizes the process of tight constraint prediction to construct a reduced problem, and defines the~\ac{tpdg} algorithm, details the learning process, and finally covers the process of achieving real-time 6-\ac{dof} powered descent guidance with~\ac{tpdg} using a Mars landing test case.

\subsection{Non-Convex Optimization}
The general non-convex continuous-time optimal control problem can be solved using a direct method by discretizing and then iteratively solving a system of equations, seeking to minimize cost while maintaining constraint satisfaction.
The general discretized non-convex optimization problem can be written as:

\textbf{Non-Convex Problem:}

\begin{equation}
\begin{aligned}
& \min_z f (z), \\
& \text{subject to:} \\
& g_i(z) = 0, \quad \forall i \in \mathcal{I}_{\text{eq}}, \\
& g_i(z) \leq 0, \quad \forall i \in \mathcal{I}_{\text{ineq}}, \\
& h_j(z) \leq 0, \quad \forall j \in \mathcal{J}_{\text{ineq}},
\label{eq: non-convex}
\end{aligned}
\end{equation}

where~\eqref{eq: non-convex} describes the nonlinear system dynamics, the non-convex state and control constraints, and the convex state and control constraints.
$\mathcal{I}_{\text{eq}} := \{1, 2, \ldots, e\}$ represents the set of non-convex equality constraint indices, $\mathcal{I}_{\text{ineq}} := \{e + 1, \ldots, p\}$ represents the set of non-convex inequality constraint indices, and $\mathcal{J}_{\text{ineq}} := \{1, 2, \ldots, q\}$ represents the set of convex inequality indices.
We assume that $g_i(z)$ and $h_j(z)$ are continuously differentiable for all $i \in I_{\text{eq}} \cup I_{\text{ineq}}$ and $j \in J_{\text{ineq}}$.
We further assume that $f(z) \in C^1$, but note that $f(z)$ can be an element of $C^0$ in practice.

\subsection{Problem Formulation: Six-Degree-of-Freedom Powered Descent Guidance}

This work aims to improve the computational efficiency of the 6-\ac{dof} powered descent guidance problem with free ignition time, including aerodynamic forces and conditional enforced constraints \cite{SzmukReynoldsEtAl2020}.
While this problem setup enables high-fidelity locally fuel-optimal trajectory generation via explicit inclusion of operational and mission constraints, directly solving the problem in real-time is challenging due to nonlinear dynamics and non-convex state and control constraints, which do not yield close-form solutions.
Furthermore, the convergence of this nonconvex problem depends heavily on both the discretization scheme and the development of a suitable initial guess trajectory to initialize the solver.

The 6-\ac{dof} powered descent guidance problem formulation used in this work assumes that speeds are sufficiently low such that planetary rotation and changes in the planet's gravitational field are negligible.
The spacecraft is assumed to be a rigid body with a constant center of mass and inertia and a fixed center of pressure.
The propulsion consists of a single rocket engine that can be gimbaled symmetrically about two axes bounded by a maximum gimbal angle $\delta_\text{max}$.
The engine is assumed to be throttleable between $T_\text{min}$ and $T_\text{max}$, remaining on until the terminal boundary conditions are met.
An atmospheric drag model is included, assuming ambient atmospheric density and constant pressure. 
The minimum-fuel 6-\ac{dof} powered descent guidance problem is formulated as follows:

\textbf{Cost Function:}
\begin{equation}
    \min_{t_f, \mathbf{T}_\mathcal{B}(t)} -m(t_f)
    \label{eq:cost}
\end{equation}

\textbf{Boundary Conditions:}
\begin{subequations}
\label{eq:boundary_conditions}
\begin{align}
    & t_f \in [0, t_{f,\text{max}}] \\
    & m(t_{0}) = m_{0}\\
    & \mathbf{r}_\mathcal{I}(t_{0}) = r_{0}  \\
    & \mathbf{v}_I(t_{0}) = v_{0} \\
    & \omega_\mathcal{B}(t_{0}) = \omega_{0}\\
    & \mathbf{r}_I(t_f) = r_f\\
    & \mathbf{v}_I(t_f) = v_f \\
    & \mathbf{q}_{\mathcal{B} \leftarrow I}(t_f) = \mathbf{q}_{f}\\
    & \omega_B(t_f) = \omega_f
\end{align}
\end{subequations}

\textbf{Dynamics:}
\begin{subequations}
\label{eq:dynamics}
\begin{align}
    & \dot{m}(t) = -\alpha_{\dot{m}} \|\mathbf{T}_{\mathcal{B}}(t) \|_2 - \beta_{\dot{m}}\\
    & \dot{\mathbf{r}}_I(t) = \mathbf{v}_I(t)\\
    & \dot{\mathbf{v}}_I(t) = \frac{1}{m(t)} \mathbf{C}_{I \leftarrow \mathcal{B}}(t) (\mathbf{T}_\mathcal{B}(t) + \mathbf{A}_\mathcal{B}(t)) + \mathbf{g}_I \\
    & \dot{\mathbf{q}}_{\mathcal{B} \leftarrow I}(t) = \frac{1}{2} \Omega_{\omega_\mathcal{B}(t)} \mathbf{q}_{B \leftarrow I}(t) \\
    & \mathbf{J}_\mathcal{B} \dot{\omega}_\mathcal{B}(t) = \mathbf{r}_{T,\mathcal{B}} \times \mathbf{T}_\mathcal{B}(t) + \mathbf{r}_{cp,\mathcal{B}} \times \mathbf{A}_\mathcal{B}(t) - \omega_\mathcal{B}(t) \times \mathbf{J}_\mathcal{B} \omega_\mathcal{B}(t)
\end{align}
\end{subequations}

\textbf{State Constraints:}
\begin{subequations}
\label{eq:state_constraints}
\begin{align}
    & m_{\text{dry}} \leq m(t) \\
    & \tan \gamma_{gs} \| H_\gamma \mathbf{r}_I(t) \|_2 \leq e_1 \cdot \mathbf{r}_I(t) \\
    & \cos \theta_{\text{max}} \leq 1 - 2 \| H_\theta \mathbf{q}_{\mathcal{B} \leftarrow I}(t) \|_2 \\
    & \|\omega_\mathcal{B}(t)\|_2 \leq \omega_{\text{max}} 
\end{align}
\end{subequations}

\textbf{Control Constraints:}
\begin{subequations}
\label{eq:control_constraints}
\begin{align}
    & 0 < T_{\text{min}} \leq \|\mathbf{T}_\mathcal{B}(t)\|_2 \leq T_{\text{max}} \\
    & \cos \delta_{\text{max}} \|\mathbf{T}_\mathcal{B}(t)\|_2 \leq e_3 \cdot \mathbf{T}_\mathcal{B}(t),
\end{align}
\end{subequations}

where the cost function in Eq.~\eqref{eq:cost} minimizes the mass at the final time, and the boundary conditions in Eq.~\eqref{eq:boundary_conditions} constrain the final time range and the initial and final states.
The dynamics in Eq.~\eqref{eq:dynamics} include mass depletion, translational state evolution, and rigid-body attitude dynamics.
Finally, the state and control constraints in Equations~\eqref{eq:state_constraints} and \eqref{eq:control_constraints} ensure that the mass is bounded, a glideslope constraint holds for landing, a maximum tilt angle is not exceeded, angular velocity is bounded, thrust is bounded, and the angle of attack is bounded.
We refer the reader to~\cite{SzmukReynoldsEtAl2020, SzmukAcikmese2018} for an extended discussion and derivation of the 6-\ac{dof} minimum fuel powered descent guidance problem

\subsubsection{Successive Convexification}\label{sec: SCvx}

\ac{scp} solves non-convex direct optimization problems by iteratively solving a sequence of local convex approximations \cite{MalyutaReynoldsEtAl2022}.
The~\ac{scp} algorithm used in this work is~\ac{scvx} due to its ability to solve non-convex constrained optimal control problems with global convergence and superlinear convergence-rate guarantees~\cite{MaoSzmukEtAl2016}.
    
\ac{scvx} solves the desired optimal control problem to optimality by successively linearizing non-convex dynamics and constraints about the initial guess provided by the iteration prior, converting the nonconvex problem into a set of convex subproblems.
Furthermore,~\ac{scvx} is guaranteed to recover local optimality if the converged solution is feasible with respect to the full problem and, if the Kurdyka-Lojasiewicz (KL) inequality holds at the converged solution, then the solution is unique~\cite{MaoSzmukEtAl2016}.

At iteration $k$,~\ac{scvx} uses a virtual control $v_i \in \mathbb{R}^{n_v}$ and $i \in 1 \dots N-1$ and the trust region $r_k \in \mathbb{R}$ to facilitate convergence.
At the $(k+1)^{th}$ iteration a convex subproblem is defined as:
    
\textbf{\ac{scvx} Convex Optimal Control Subproblem:}
\begin{equation}
\begin{aligned}
& \min_{d, w} \quad & L^k(d, w) \\
& \text{subject to:} \quad & u^k + w &\in U, \\
& & x^k + d &\in X, \\
& & ||w|| &\leq r^k,
\end{aligned}
\label{eq: subproblem}
\end{equation}

where $U$ and $X$ are convex sets, often second-order cones, $L^k(d, w)$ is the penalized cost function, $(x^k, u^k)$ are the current state and control iterates, $d_i := x_i - x^k_i$, and $w_i := u_i - u^k_i$.
The iterates $(x^k, u^k)$ are then successively updated using the linearized dynamics, which include virtual control and virtual buffer zones to mitigate artificial infeasibility:

\begin{equation}
    v := [v_1^T, v_2^T, \dots, v_{N-1}^T]^T \in \mathbb{R}^{n_v (N-1)},
    \label{eq: virtual control}
\end{equation}

\begin{equation}
    s' := [s'^T_1, s'^T_2, \dots s'^T_{N-1}]^T \in \mathbb{R}_+^{n_s (N-1)}.
    \label{eq: virtual buffer}
\end{equation}

The virtual control term $v$ is left unconstrained, so any state in the feasible region of the convex subproblem is reachable in finite time, preventing artificial infeasibility after linearization.
The virtual buffer zone $s'_i \in \mathbb{R}_+^{n_s}$ similarly maintains state reachability for the linearized non-convex state and control constraints.
The first-order approximation of the dynamics and nonconvex constraints is then:

\begin{equation}
    x_{i+1}^k + d_{i+1} = f(x_i^k, u_i^k) + A_i^k d_i + B^k_i w_i + E^k_i v_i,
    \label{eq: dynamics linear}
\end{equation}

\begin{equation}
    s(x_i^k, u_i^k) + S_i^k d_i + Q^k_i w_i - s'_i \leq 0,
    \label{eq: nonconvex constraints linear}
\end{equation}

where $A_i^k := \frac{\partial}{\partial x_i} f(x_i, u_i)|_{x_i^k, u_i^k}$, $B_i^k := \frac{\partial}{\partial u_i} f(x_i, u_i)|_{x_i^k, u^k_i}$, $S_i^k := \frac{\partial}{\partial x_i} s(x_i, u_i)|_{x_i^k, u^k_i}$, $Q_i^k := \frac{\partial}{\partial u_i} s(x_i, u_i)|_{x_i^k, u^k_i}$, and $E_i^k \in \mathbb{R}^{n_x \times n_v}$ such that $\text{im}(E_i^k) = \mathbb{R}^{n_x}$. Algorithm~\ref{alg:SCvx} details the full~\ac{scvx} algorithm.

\begin{algorithm}[H]
\caption{The SCvx Algorithm}
\label{alg:SCvx}
\begin{algorithmic}[1]
\Procedure{SCvx}{$x^1, u^1, \lambda, \epsilon_{\text{tol}}$}
    \State \textbf{input:} Select initial state $x^1 \in X$ and control $u^1 \in U$. Initialize trust region radius $r^1 > 0$. Select penalty weight $\lambda > 0$, and parameters $0 < \rho_0 < \rho_1 < \rho_2 < 1$, $r_l > 0$ and $\alpha > 1, \beta > 1$.
    \While{not converged, i.e., $\Delta J^k > \epsilon_{\text{tol}}$}
        \State \textbf{step 1:} At $(k + 1)$-th succession, solve Equation \ref{eq: subproblem} at $(x^k, u^k, r^k)$ to get an optimal solution $(d, w)$.
        \State \textbf{step 2:} Compute the actual change in the penalty cost:
        \begin{equation} \label{eq:actual_change}
            \Delta J^k = J(x^k, u^k) - J(x^k + d, u^k + w),
        \end{equation}

        where $J(x,u) := C(x,u) + \sum_{i=1}^{N-1} \lambda_i P(x_{i+1} - f(x_i, u_i), s(x_i, u_i))$.
        
        \State and the predicted change by the convex cost:
        \begin{equation} \label{eq:predicted_change}
            \Delta L^k = J(x^k, u^k) - L^k(d, w),
        \end{equation}

        where $L^k(d,w) := C(x^k + d, u^k + w) + \sum_{i=1}^{N-1} \lambda_i P(E^k_i v_i, s'_i)$.
        
        \If{$\Delta J^k = 0$}
            \State stop, and return $(x^k, u^k)$;
        \Else
            \State compute the ratio
            \begin{equation} \label{eq:rho}
                \rho^k := \frac{\Delta J^k}{\Delta L^k}.
            \end{equation}
        \EndIf
        \State \textbf{step 3:}
        \If{$\rho^k < \rho_0$}
            \State reject this step, contract the trust region radius, i.e., $r^k \leftarrow \frac{r^k}{\alpha}$ and go back to step 1;
        \Else
            \State accept this step, i.e., $x^{k+1} \leftarrow x^k + d$, $u^{k+1} \leftarrow u^k + w$, and update the trust region radius $r^{k+1}$ by
            \begin{equation}
                r^{k+1} =
                \begin{cases}
                    \frac{r^k}{\alpha}, & \text{if } \rho^k < \rho_1; \\
                    r^k, & \text{if } \rho_1 \leq \rho^k < \rho_2; \\
                    \beta r^k, & \text{if } \rho_2 \leq \rho^k.
                \end{cases}
                \label{eq: trust region update}
            \end{equation}
        \EndIf
        \State $r^{k+1} \leftarrow \max\{r^{k+1}, r_l\}$, $k \leftarrow k + 1$, and go back to step 1.
    \EndWhile
    \State \Return $(x^{k+1}, u^{k+1})$.
\EndProcedure
\end{algorithmic}
\end{algorithm}

Following the guidance in~\cite{MaoSzmukEtAl2016}, heuristic recommendations for parameter choice include choosing $\rho_0 \approx 0$, $\rho_1 \gtrsim 0$, $\rho_2 \lesssim 1$, $r_l \gtrsim 0$.
Equation~\eqref{eq:rho} compares the realized nonlinear cost reduction $\triangle J^k$ to the predicted linear cost reduction $\triangle L^k$ using the previous, $k^{th}$, iteration.
If the linear cost reduction over-predicts the realized nonlinear cost reduction, $\rho^k < \rho_0 \ll 1$, the step is rejected.
Otherwise, the step is accepted, either contracting, maintaining, or growing the trust region size according to Eq.~\eqref{eq: trust region update}; shrinking the trust region when the linearization accuracy is acceptably deficient and growing when the linear cost reduction under-predicts or almost under-predicts the realized nonlinear cost reduction.
Recent work has extended trust region updates to include state-dependent trust regions based on a nonlinearity index defined by the state transition matrix~\cite{BernardiniBaresiEtAl2024}.
We note that a feasible initial guess is not required for~\ac{scvx} and an infeasible initial guess does not affect the algorithm's convergence guarantees~\cite{MaoSzmukEtAl2016}.

In this work,~\ac{scvx} is applied to a discretized system, where the decision vector $z$ is finite-dimensional and the exact penalty function used in Eq.\eqref{eq:actual_change} is defined as

\begin{equation}
    J(z) := g_0(z) + \sum_{i \in \mathcal{I}_{\text{ncvx eq}}} \lambda_i |g_i(z)| + \sum_{i \in \mathcal{I}_{\text{ncvx ineq}}} \lambda_i \max(0, g_i(z)) + \sum_{i \in \mathcal{I}_{\text{cvx ineq}}} \tau_i \max(0, h_j(z)),
    \label{eq: exact penalty function}
\end{equation}

where the cost $g_0(z) \in C^1$, nonconvex equality and inequality constraints $g_i(z)$ and convex inequality constraints $h_i(z)$ are continuously differentiable for all $i \in \mathcal{I}_{\text{ncvx eq}} \cup \mathcal{I}_{\text{ncvx ineq}} \cup \mathcal{I}_{\text{cvx ineq}}$, and penalty weights $\lambda_i \geq 0$ and $\tau_i \geq 0$.

The only assumption required to achieve global weak convergence is the~\ac{licq} : the columns of the matrices that make up the gradients of the active constraints are linearly independent at the local optimum~\cite{MaoSzmukEtAl2016}.
If the~\ac{licq} is satisfied, Algorithm~\ref{alg:SCvx} always has limit points and any limit point $\bar{z}$ is a stationary point of the non-convex penalty function, Equation~\eqref{eq: exact penalty function}.
Furthermore, if $\bar{z}$ is feasible for the original non-convex problem, then it is a~\ac{kkt} point of the problem.

To achieve strong convergence, where single limit point convergence is guaranteed, Lipschitz continuous gradients are required for the set of inequality constraints: $\exists L_i \geq 0$, such that $||\nabla g_i(z_2) - \nabla g_i (z_1)|| \leq L_i ||z_2 - z_1||$ for all $i = 0, \dots, p+q$, and the penalized cost function $J(z)$ must have the KL property~\cite{AttouchBolteEtAl2013}.
If both these conditions hold, along with the~\ac{licq}, the sequence $\{z^k\}$ generated by~\ac{scvx} always converges to a single limit point $\bar{z}$, and $\bar{z}$ is a stationary point of the non-convex penalty $J(z)$.
Additionally, if $\bar{z}$ is feasible for the original non-convex problem, then it is a~\ac{kkt} point of that problem. 

Finally, to maintain the superlinear convergence rate, strict complementary slackness must be met at the local optimum $\bar{z}$: $\lambda_i > 0$ for all active constraints 
and for $z^k \rightarrow \bar{z}$, there are at least $n_u (N-1)$ binding constraints:

\begin{equation}
    |\mathcal{A}_{\text{ncvx eq}}| + |\mathcal{A}_{\text{ncvx ineq}}| + |\mathcal{A}_{\text{cvx ineq}}| \geq n_u (N-1),
    \label{eq: binding constraints}
\end{equation}

where $\mathcal{A}$ is the set of active, or tight, constraints.
While bang-bang control solutions, which hold for the 3-\ac{dof} powered descent guidance problem, always uphold the binding assumption, the 6-\ac{dof} extension, which includes rotational dynamics and aerodynamics forces, is not guaranteed to yield this type of control solution.
If both of these assumptions hold, along with the~\ac{licq}, superlinear convergence is theoretically guaranteed.
Recent extensions of~\ac{scvx} include penalty-based reformulations of path constraints, generalized time-dilation, multiple-shooting discretization, $l_1$ exact penalization of the nonconvex constraints, and the prox-linear method for convex-composite minimization~\cite{ElangoLuoEtAl2024}.
For a full derivation and description of the~\ac{scvx} algorithm and its convergence properties, see~\cite{MaoSzmukEtAl2016}.

\subsection{Tight Constraint Prediction}
Consider the parametric formulation of the powered descent guidance problem, where the parameter vector $\theta \in \Theta \subseteq \mathbb{R}^{n_p}$ is drawn from a representative set of parameters and maps to a binary vector $\tau(\theta) \in \{0,1\}^M$.
Here, ones denote tight or active constraints at the indicated discretization nodes and zeros denote non-tight or inactive constraints.
If the optimization problem is non-degenerate, then the tight constraints serve as the support constraints for the optimization problem, and removing any of the tight constraints would result in a decrease in the objective function for minimization problems~\cite{Calafiore2010}.
\ac{tpdg} efficiently learns the mapping $\theta \rightarrow \tau(\theta)$ using prior runs of the constrained optimization problem to accurately predict the set of tight or active constraints at the globally (or locally) optimal solution.

For an inequality-constrained optimization problem (Equation~\eqref{eq: non-convex}), the first order necessary conditions (\ie{}, \ac{kkt} conditions), must hold.

\begin{definition}{First Order Necessary Conditions: KKT Conditions}
    
    \[
    \mathcal{L}(x, \lambda) = f(x) - \sum_{i \in \mathcal{I} \cup \mathcal{E}} \lambda_i c_i(x)
    \]
    
    \begin{enumerate}
        \item $\nabla_x \mathcal{L}(x^*, \lambda^*) = 0$
        \item $c_i(x^*) = 0, \quad i \in \mathcal{E}$
        \item $c_i(x^*) \geq 0, \quad i \in \mathcal{I}$
        \item $\lambda_i^* \geq 0, \quad i \in \mathcal{I}$
        \item $\lambda_i^* c_i(x^*) = 0, \quad i \in \mathcal{I} \cup \mathcal{E}$
    \end{enumerate}

\end{definition}

The active $\mathcal{A}$ set at any feasible $x$ consists of the set of the equality constraint indices $\mathcal{E}$

\begin{definition}{Active Set}
    \[
    \mathcal{A}(x) = \mathcal{E} \cup \{ i \in \mathcal{I} \mid c_i(x) = 0 \}
    \]
    
\end{definition}

\begin{claim}
    When given the tight/active constraints set, the reduced problem defined by only the equality and active inequality constraints will recover the optimal solution for the original problem.
\end{claim}

\begin{proof}
    
    Assume there exists a feasible sequence $\{ z_k \}$ such that $\nabla f(x^*)^T d < 0$. Then, for the limiting direction $d$:

    \begin{equation}
    \lim_{z_k \in s_d} \frac{z_k - x}{\| z_k - x \|} \rightarrow d,
    \label{eq: limiting direction}
    \end{equation}

    where $s_d$ is some subsequence. Then, by the Taylor series expansion, there exists a limiting direction $d$ such that 

    \[
    f(z_k) = f(x^*) + (z_k - x^*)^T \nabla f(x^*) + o(\|z_k - x^*\|).
    \]
    
    Substituting in the limiting direction expression from Equation \ref{eq: limiting direction}, $(z_k - x^*)^T = \|z_k - x^*\| d^T$:
    
    \[
    f(z_k) = f(x^*) + \|z_k - x^*\| d^T \nabla f(x^*) + o(\|z_k - x^*\|),
    \]

    and we observe $\|z_k - x^*\| > 0$ and $o(\|z_k - x^*\|) > 0$. Further, $d^T \nabla f(x^*) < 0$, from our assumption.
    
    From the~\ac{kkt} conditions, there exists a local solution $x^* \in \mathcal{O} \quad \text{for} \quad k > K \in \mathbb{N}$ such that

    \[
    f(z_k) < f(x^*) + \frac{1}{2} \|z_k - x^*\| \cdot d^T \nabla f(x^*).
    \]

    Since $d^T \nabla f(x^*) \geq 0 \quad \text{happens when} \quad c_i(x^*) = 0$ and the active set is defined by $\mathcal{A} = \{ i \in \mathcal{I} \mid c_i(x^*) = 0 \} \cup \mathcal{E}$, the set of tight/active inequality constraints and the equality constraints hold at the optimal solution. Therefore, solving the reduced problem defined by constraints in the active set $\mathcal{A}$ achieves the same solution as solving the full problem.

\end{proof}

\subsection{Transformer-based Successive Convexification}

In this work, we extend~\ac{tpdg} to the Successive Convexification algorithm, resulting in the Transformer-based Successive Convexification algorithm. We enhance the computational and data efficiency for nonconvex-powered descent guidance by employing symmetry-invariant data augmentation.
\ac{tpdg} was previously applied to the 3-\ac{dof} minimum fuel powered descent guidance problem, improving solution times by up to an order of magnitude compared to~\ac{lcvx}~\cite{BridenGurgaEtAl2024}.
By learning to predict the set of tight or active constraints at the optimal solution,~\ac{tpdg} creates the minimal reduced-size problem initialized with only the tight constraints, then uses the solution of this reduced problem to warm-start the direct optimization solver.
6-\ac{dof} powered descent guidance is known to be challenging to solve quickly and reliably due to the nonlinear and non-convex nature of the problem, the discretization scheme heavily influencing solution validity, and reference trajectory initialization determining algorithm convergence or divergence.
Our contributions in this work address these challenges by extending~\ac{tpdg} to learn the set of tight constraints for the~\ac{scvx} formulation of the 6-\ac{dof} powered descent guidance problem. 
In addition to reducing the problem size, feasible and locally optimal reference trajectories are also learned to facilitate convergence from the initial guess.
\ac{tpdg} enables onboard computation of real-time guidance trajectories, demonstrated by a 6-\ac{dof} Mars powered landing application problem.

Algorithm~\ref{alg:SCvx} is modified to include an inference step that predicts the set of tight constraints for every convex sub-problem based on the problem parameters and iteration number $k$. 
Furthermore, the change in the predicted number of tight constraints is included in the trust region contraction or growth computation to allow the trust region to respond not only to nonlinearity but also to changes in problem structure.
The~\ac{tscvx} algorithm is presented in Algorithm~\ref{alg:TSCvx}.

\begin{algorithm}[H]
\caption{The T-SCvx Algorithm}
\label{alg:TSCvx}
\begin{algorithmic}[1]
\Procedure{T-SCvx}{$x^1, u^1, \lambda, \epsilon_{\text{tol}}, f_1, f_2$}
    \State \textbf{input:} Predict the solution using the solution prediction Transformer NN $z = f_1(\theta)$ and select initial state with the prediction $(x^1, u^1) = z$. Initialize trust region radius $r^1 > 0$. Select penalty weight $\lambda > 0$, and parameters $0 < \rho_0 < \rho_1 < \rho_2 < 1$, $r_l > 0$ and $\alpha > 1, \beta > 1$.
    \While{not converged, i.e., $\Delta J^k > \epsilon_{\text{tol}}$}
        \State \textbf{step 1:} At $(k + 1)$-th succession, predict the set of tight constraints using the trained transformer neural network $\tau = f_2(\theta, k+1)$. Define the subproblem, Equation \ref{eq: subproblem}, using only the constraints identified to be tight.
        \State \textbf{step 2:} Solve Equation \ref{eq: subproblem} at $(x^k, u^k, r^k)$ to get an optimal solution $(d, w)$.
        \State \textbf{step 3:} Compute the actual change in the penalty cost:
        \begin{equation} \label{eq:actual_change}
            \Delta J^k = J(x^k, u^k) - J(x^k + d, u^k + w),
        \end{equation}

        where $J(x,u) := C(x,u) + \sum_{i=1}^{N-1} \lambda_i P(x_{i+1} - f(x_i, u_i), s(x_i, u_i))$.
        
        \State and the predicted change by the convex cost:
        \begin{equation} \label{eq:predicted_change}
            \Delta L^k = J(x^k, u^k) - L^k(d, w),
        \end{equation}

        where $L^k(d,w) := C(x^k + d, u^k + w) + \sum_{i=1}^{N-1} \lambda_i P(E^k_i v_i, s'_i)$.
        
        \If{$\Delta J^k = 0$}
            \State stop, and return $(x^k, u^k)$;
        \Else
            \State compute the ratio
            \begin{equation} \label{eq:rho}
                \rho^k := \frac{\Delta J^k}{\Delta L^k}.
            \end{equation}
        \EndIf
        \State \textbf{step 4:} Compute the percentage of changed constraints $\tau_r = \frac{\sum |\tau_{k+1} - \tau_{k}|}{\text{len}(\tau_{k+1})}$.
        \State \textbf{step 5:}
        \If{$\rho^k < \rho_0$}
            \State reject this step, contract the trust region radius, i.e., $r^k \leftarrow \frac{r^k}{\alpha^{\tau_r}}$ and go back to step 1;
        \Else
            \State accept this step, i.e., $x^{k+1} \leftarrow x^k + d$, $u^{k+1} \leftarrow u^k + w$, and update the trust region radius $r^{k+1}$ by
            \begin{equation}
                r^{k+1} =
                \begin{cases}
                    \frac{r^k}{\alpha^{\tau_r}}, & \text{if } \rho^k < \rho_1; \\
                    r^k, & \text{if } \rho_1 \leq \rho^k < \rho_2; \\
                    \beta^{(1-\tau_r)} r^k, & \text{if } \rho_2 \leq \rho^k.
                \end{cases}
                \label{eq: trust region update}
            \end{equation}
        \EndIf
        \State $r^{k+1} \leftarrow \max\{r^{k+1}, r_l\}$, $k \leftarrow k + 1$, and go back to step 1.
    \EndWhile
    \State \Return $(x^{k+1}, u^{k+1})$.
\EndProcedure
\end{algorithmic}
\end{algorithm}

\subsubsection{Transformer Neural Network Architecture}\label{sec: Transformer NN}

The model structure for T-SCvx is considered for the following prediction problem: given a set of parametric inputs for a constrained optimization problem, $\{ \theta_1, ..., \theta_L)\}$, we would like to predict the set of tight constraints, $\{ \tau(\theta_1), ..., \tau(\theta_L) \}$, where each $\tau(\theta_i)$ is an 1 x M matrix (M is the number of inequality constraints in the original problem)~\cite{BridenGurgaEtAl2024}.
Next, linear encoder and decoder layers transfer the input data into a higher dimensional embedding space and the output into a lower dimensional output space.
A learned position encoding is then applied to preserve the temporal order of the input data.
From the position encoder, a transformer encoder with several heads, $h$, uses multi-head attention to transform the data into query matrices, $Q_h^{(i)}$, key matrices, $K_h^{(i)}$, and value matrices, $V_h^{(i)}$.
Finally, the attention output is generated by scaled production, as shown in Eq.~\eqref{eq: Transformer eqn}.

\begin{equation}
(O_h^{(i)})^T = \text{Attention} \left( Q_h^{(i)}, K_h^{(i)}, V_h^{(i)} \right) = \text{Softmax} \left(\frac{Q_h^{(i)} K_h^{(i) T}}{\sqrt{d_k}} \right) V_h^{(i)}
\label{eq: Transformer eqn}
\end{equation}

Additional linear layers, dropout, and LayerNorm layers are also present in the transformer encoder layer.
The full model was designed in PyTorch using the~\texttt{torch.nn} module~\cite{PaszkeGrossEtAl2017}.
The implemented tight constraints NN for the 6-\ac{dof}~\ac{pdg} application has $1 \times 17$-dimensional problem parameter inputs, including the initial velocity, initial position, angular velocity, initial quaternion, initial mass, pitch angle, glideslope angle, and iteration number, extended from the model trained in~\cite{BridenGurgaEtAl2024}.
We note that the final position and velocity are set to zero as this application is a powered descent landing problem and, consequently, reference frames can be adjusted accordingly for a varying final position.
Additional parameters for the problem can be included and would only result in a larger input size and a potentially larger required neural network architecture.
Furthermore, planetary and spacecraft design parameters were kept constant to represent the chosen mission design.
Since state and constraint parameters may change during operation, these variables were chosen as the parameters for the parametric optimization problem.

\subsection{T-SCvx Algorithm}

Algorithm~\ref{alg:t-pdg} describes the procedure for applying the transformer NNs for problem reduction in real-time.
First, the NN models for predicting tight constraints are called to generate the strategy at iteration $k+1$.
Using the strategy, the solver is called to find the corresponding solution and cost, as determined by the problem's cost function.
Since the full problem is used for the evaluation cost function, T-SCvx will not return a solution unless it is both optimal and constraint satisfying for the full problem.

\begin{algorithm}
  \caption{T-SCvx}\label{alg:t-pdg}
  \begin{algorithmic}[1]
    \Procedure{T-SCvx}{$\theta$, k+1}
        \State $strategy \gets \Call{NN-Prediction}{\mathrm{tight\_constraints\_model}, \theta, k+1}$ \Comment{Predict the optimal strategy}
        \State $(soln, cost) \gets \Call{Reduced-Solve}{\theta, strategy}$ 
        \State \textbf{return} $soln$
    \EndProcedure
  \end{algorithmic}
\end{algorithm}

To acquire the samples required for training, we use the~\ac{scvx} algorithm from the SCP Toolbox, which uses the ECOS solver~\cite{DomahidiChuEtAl2013}\footnote{https://github.com/jump-dev/ECOS.jl} \cite{MalyutaReynoldsEtAl2022}.
A custom implementation of the free final time 6-\ac{dof} minimum fuel powered descent guidance problem presented in Equations~\eqref{eq:cost}-\eqref{eq:control_constraints} was programmed in Julia.
Utilizing the solver, we efficiently sample feasible regions of the solution space to accelerate the convergence of the optimization process.
Numerical scaling was used for variable ranges to ensure the solver was well-conditioned.
Because the 6-\ac{dof} powered descent guidance problem is non-convex, advanced sampling methods must be used to avoid computational inefficiencies by avoiding non-feasible regions in the solution space.
    
\subsubsection{Symbolic Implementation of the 6-DoF Problem}
To utilize~\ac{scvx} to iteratively solve the non-convex optimization problem, the problem must be linearized and discretized.
The dynamics, kinematics, and constraints are linearized through the Jacobian matrix.
The state vector $\mathbf{x}(t)$ and control vector $\mathbf{u}(t)$ are defined as follows:\\
\begin{equation}
    \mathbf{x}(t) = \begin{bmatrix}
        \mathbf{r}(t) & \mathbf{v}(t) & \mathbf{q}(t) & \boldsymbol{\omega}(t) & m(t)
    \end{bmatrix}^T
\end{equation}
\begin{equation}
    \mathbf{u}(t) = \mathbf{T}(t)
\end{equation}
The Jacobian matrix of the convex constraints defined in Equation \ref{eq:dynamics} with respect to the state vector and control vector is defined as:
\begin{equation}
    \mathbf{A}(t) = \frac{\partial \mathbf{f}}{\partial \mathbf{x}}, \quad \mathbf{B}(t) = \frac{\partial \mathbf{f}}{\partial \mathbf{u}}
\end{equation}
The partial derivatives required to calculate the matrix in 6-\ac{dof} rely on quaternion and skew-symmetric matrix derivative calculations and thus require symbolic differentiation to derive the matrix analytically.
Our method uses Symbolics package in Julia to define the partial derivatives efficiently and avoid numerical errors in finite difference methods.
The partial derivatives are calculated at each time step and assembled into the Jacobian matrices $\mathbf{A}(t)$ and $\mathbf{B}(t)$.
The process is repeated at each time step in the successive convexification process as laid out in Equation~\eqref{eq: subproblem}.
    
\subsubsection{Data Sampling Strategy}\label{sec: T-SCvx sampling}
Since the 6-\ac{dof} powered descent guidance problem is non-convex, using a naive uniform sampling algorithm like~\ac{tpdg} in 3-\ac{dof} uses would be too computationally inefficient to generate a sizable dataset for training.
Therefore, our sampling strategy utilizes the symmetry of the system dynamics and constraint activation to efficiently sample over a wide range of initial conditions.
Every initial sample was sampled with only a positive range of East and North initial condition coordinates.
Next, a rotation about the Up axis was computed at the angles $0^\circ$, $45^\circ$, $90^\circ$, $135^\circ$, $180^\circ$, $225^\circ$, $270^\circ$, and $315^\circ$.
We note that all induced constraints are symmetric about the Up axis, so the set of tight constraints that are mapped to remain invariant under these rotations.
By synthetically rotating the dataset around the up axis, sampling efficiency is achieved while encoding symmetries into the NN.
In this work, we sampled a dataset of tight constraints and optimal solutions consisting of less than 1,600 samples.
After performing the set of rotations the dataset reached over 11,600 samples.
Further, this method of data augmentation has been shown to have similar accuracies to rotational invariant NN architectures~\cite{QuirogaRonchettiEtAl2019}.
    
\subsubsection{Training, Validation, and Testing}\label{sec: Training T-SCvx}
The transformer NN architecture (described in Section~\ref{sec: Transformer NN}) implemented for~\ac{tpdg} was utilized for both the tight constraints prediction and solution prediction NNs in~\ac{tscvx}.
All sampled data, including the rotated samples, were split into 80\% training and validation data and 20\% test data.
The datasets were standardized by subtracting the mean and dividing by the standard deviation and a K-fold training process was utilized with $K = 3$, 4000 warmup steps, and two epochs.
MSE loss was used for training and testing the solution neural network, while binary loss was used for testing the tight constraint prediction neural network.
Applying~\ac{tpdg} to 6-\ac{dof} powered descent guidance improves the initial guess for SCP algorithms and can indicate the predicted number of active constraints.
Furthermore,~\ac{tscvx} efficiently determines if superlinear convergence holds for the problem (see Section~\ref{sec: SCvx}).

\section{Transformer-based Successive Convexification}\label{sec: 6 DoF TPDG}

\subsection{Problem Setup and Parameters}
The problem parameters are shown in Table~\ref{tab:problem_parameters_TSCvx}, where the units are dimensionless quantities of length $U_L$, time $U_T$, and mass $U_M$ to improve numerical scaling for the solver.
The same problem parameters were used for successive convexification, except for the trust region initialization, $\eta_\text{init}$, where a smaller trust region was used for~\ac{tscvx} since we assume we are already very close to the final solution when an initial guess is predicted, in addition to ensuring trust region conservatism in the case the guess or tight constraints were not fully accurate.
A total of 1,592 samples were computed using the~\ac{scvx} algorithm (Algorithm~\ref{alg:SCvx}), and rotations were then computed on the data, as described in Section \ref{sec: T-SCvx sampling}, to create a dataset of 11,634 samples.
Throughout this study,~\ac{scvx} was limited to 20 iterations maximum.
If the algorithm does not converge when the maximum iteration number is met, the problem is considered infeasible.
The parameters sampled over include $\theta = \gamma_{gs}, \theta_{\text{max}}, r_0, v_0, q_0, \omega_0, m_0$, including the glideslope angle, maximum pitch angle, initial position, initial velocity, initial quaternion, initial angular velocity, and initial mass.
Representing the full initial state for 6-\ac{dof} motion and constraint parameters which may vary between the parachute and powered descent initiation phase.
An additional extension from 3-\ac{dof}~\ac{tpdg} includes the addition of the mass parameter, allowing multiple trajectories to be computed during the descent process instead of assuming only one trajectory is generated at powered descent initiation.
For the tight constraint prediction NN, an additional parameter of the~\ac{scvx} iteration number, $k$, is included so the NN can predict the tight constraints for a given iteration number.
The final state is maintained constant since a change in the reference frame can be used to represent any desired configuration.
The strategy is a binary vector $\tau(\theta)$ corresponding to the inequality constraints (\ref{eq:state_constraints}-\ref{eq:control_constraints}).
We note that the implementation of~\ac{tscvx} (Algorithm~\ref{alg:TSCvx}) only uses the convex constraints to reduce the convex subproblems, but an extension could be easily made to reduce the non-convex cost function based on the predicted non-convex tight constraints.
Instead of only a final time prediction neural network,~\ac{tscvx} maps to the full discretized solution, including the full state, control, and final time parameter $(x, u, t_f)$, and this enables an accurate initial guess to be generated.

\begin{table}[h!]
\centering
\caption{Problem Parameters}

\begin{tabular}{c c c }
\hline
\hline
\textbf{Param.} & \textbf{Value} & \textbf{Units} \\ 
\hline
g\(_I\) & -\(\mathbf{e}_1\) & \(U_L/U^2_T\) \\ 
\(\rho\) & 0.020 & \(U_M/U^3_L\) \\ 
\(\mathbf{J}_B\) & 0.01 \(\textbf{diag}[0.1 \; 1 \; 1]\) & \(U_M \cdot U^2_L\) \\ 
\(P_{\text{amb}}\) & 0.1 & \(U_M/U^2_T/U_L\) \\ 
\(A_{\text{noz}}\) & 0.5 & \(U^2_L\) \\ 
\(r_{\text{cp},B}\) & \(0_{3 \times 1}\) & \(U_L\) \\ 
\(r_{T,B}\) & -0.01\(\cdot e_1\) & \(U_L\) \\ 
\(\rho S_A C_A\) & 0.2 & - \\
\(I_{\text{sp}}\) & 30.0 & \(U_T\) \\ 
\(\omega_{\text{max}}\) & 90.0 & \(^\circ/U_T\) \\ 
\(\delta_{\text{max}}\) & 20.0 & \(^\circ\) \\ 
\(T_{\text{min}}\) & 0.3 & \(U_M \cdot U_L / U^2_T\) \\ 
\(T_{\text{max}}\) & 5.0 & \(U_M \cdot U_L / U^2_T\) \\ 
V\(_\alpha\) & 2.0 & \(U_L/U_T\) \\ 
\(m_{\text{dry}}\) & 2.0 & \(U_M\) \\ 
\(\mathbf{r}_{I,N}\) & \(0_{3 \times 1}\) & \(U_L\) \\ 
\(\mathbf{v}_{I,N}\) & -0.1\(\cdot e_1\) & \(U_L/U_T\) \\ 
\(\omega_{\mathcal{B},1},\omega_{\mathcal{B},N}\) & \(0_{3 \times 1}\) & \(^\circ/U_T\) \\ 
\(\mathbf{q}_{B \leftarrow I,N}\) & \(q_{\text{id}}\) & - \\ 
\(N\) & 50 & - \\ 
\(N_\text{sub}\) & 49 & - \\ 
\(\text{iter}_{\text{max}}\) & 20 & - \\ 
\(\text{disc}\; \text{method}\) & FOH & - \\ 
\(\lambda\) & 500 & - \\ 
\(\rho_0\) & 0 & - \\ 
\(\rho_1\) & 0.1 & - \\ 
\(\rho_2\) & 0.7 & - \\ 
\(\beta_{sh}\) & 2.0 & - \\ 
\(\beta_{gr}\) & 2.0 & - \\ 
\(\eta_{\text{full, init}}\) & 2.0 & - \\ 
\(\eta_{\text{reduced, init}}\) & 0.01 & - \\ 
\(\eta_{lb}\) & 0.001 & - \\ 
\(\eta_{ub}\) & 10.0 & - \\ 
\(\epsilon_{abs}\) & 0.1 & - \\ 
\(\epsilon_{rel}\) & 0.001 & - \\ 
\(\text{feas} \; \text{tol}\) & 0.5 & - \\ 
\(q_{tr}\) & Inf & - \\ 
\(q_{exit}\) & Inf & - \\ 
solver & ECOS & - \\ 
solver maxit & 1000 & - \\ 
\end{tabular}
\label{tab:problem_parameters_TSCvx}
\end{table}

The NNs for tight constraint and solution prediction used the transformer architecture in Section~\ref{sec: Transformer NN}.
The constraint prediction NN has an input size of 17 and an output size of 12N, using an architecture with 384-dimensional layers, two heads, four layers, and 0.1 dropout.
The solution prediction NN, was designed to be larger since the input size of 16 corresponded to an output size of 17N+1. 768-dimensional layers were used, with two heads, four layers, and 0.1 dropout.
The range of values sampled for NN training is included in Table~\ref{tab:sampledTSCvx}.

\begin{table}
    \centering
    \caption{\label{tab:sampledTSCvx} Sampled Dataset}
    \small
    \begin{tabular}{>{\raggedright\arraybackslash}p{2.5cm} >{\raggedright\arraybackslash}p{2.0cm} >{\raggedright\arraybackslash}p{2.0cm}}
        \hline
        \textbf{Parameter} & \textbf{Min Range} & \textbf{Max Range} \\
        \hline
        K & 1 & 20 \\
        Glideslope angle ($\gamma_{gs}$) & 0 deg & 90 deg \\
        Pitch angle ($\theta_\text{max}$) & 0 deg & 359.4 deg \\
        Initial Position ($r_0$) & 0.003 $\hat{e}_x +$ -13.83 $\hat{e}_y +$ -13.83 $\hat{e}_z$ & 9.997 $\hat{e}_x +$ 13.83 $\hat{e}_y +$ 13.83 $\hat{e}_z$ \\
        Initial Velocity ($v_0$) & -1.998 $\hat{e}_x +$ -2.779 $\hat{e}_y +$ -2.779 $\hat{e}_z$ & -0.001 $\hat{e}_x +$ 2.779 $\hat{e}_y +$ 2.779 $\hat{e}_z$ \\
        Initial Quaternion ($q_0$) & -0.996 + -1 $i +$ -1 $j +$ -1 $k$ &  1 + 1 $i +$ 1 $j +$ 0.999 $k$ \\
        Initial Angular Velocity ($\omega_0$) & -89.98 $\hat{e}_x +$ -123.49 $\hat{e}_y +$ -123.49 $\hat{e}_z$ deg & 89.70 $\hat{e}_x +$ 123.49 $\hat{e}_y +$ 123.49 $\hat{e}_z$ deg \\
        Initial mass ($m_0$) & 0.003 & 3 \\
        \hline
    \end{tabular}
\end{table}

\subsection{Results and Analysis}
Utilizing the training process described in Section~\ref{sec: Training T-SCvx}, the tight constraint prediction and solution NNs were trained with less than 12,000 samples, most of which were produced using efficient rotation-based data augmentation.
Both NNs converged within two epochs of training for batch sizes of 128. Table~\ref{tab:trainandtestTSCvx} shows the training and validation MSE at the third fold of training and the accuracy of each NN on the test dataset.
Since the predicted constraints are binary, binary accuracy was used for test time evaluation, while MSE was used for the floating point solution prediction values. 

\begin{table}[ht]
    \caption{Training and Testing of~\ac{tscvx} with Baseline Comparisons}
    \label{tab:trainandtestTSCvx}
    \centering
    \small
    \begin{tabular}{>{\raggedright\arraybackslash}p{2.5cm} 
                    >{\centering\arraybackslash}p{2.5cm} 
                    >{\centering\arraybackslash}p{2.5cm} 
                    >{\centering\arraybackslash}p{3cm} 
                    >{\centering\arraybackslash}p{2cm}}
    \hline
    Model & Train (MSE) & Validation (MSE) & Test (MSE / Binary Acc.) & \# of Params \\
    \hline
    Constraint NN & 0.024 & 0.024 & 96.45\% (Binary Acc.) & 8.91M \\
    Solution NN & 0.975 & 1.092 & 1.040 (MSE) & 9M \\
    \hline
    Predict Only Zeros & - & - & 95.95\% (Binary Acc.) & - \\
    Predict Only Ones & - & - & 4.05\% (Binary Acc.) & - \\
    \hline
    \end{tabular}
\end{table}

The results from~\ac{tscvx}, from Algorithm~\ref{alg:TSCvx}, compared to~\ac{scvx}, using 545 samples from the test dataset, are shown in Figure \ref{fig:solveTimesSCvx}.
From the test results,~\ac{tscvx} reduces the mean solve time for the 50-timestep 6-\ac{dof} powered descent guidance problem by 66\% and the median solve time by 70\%.
The standard deviation for~\ac{tscvx} is slightly higher than~\ac{scvx}, at 5.24 seconds, but one standard deviation remains below the~\ac{scvx} mean of 14.61.

\begin{figure}
    \centering
    \includegraphics[width=0.8\textwidth]{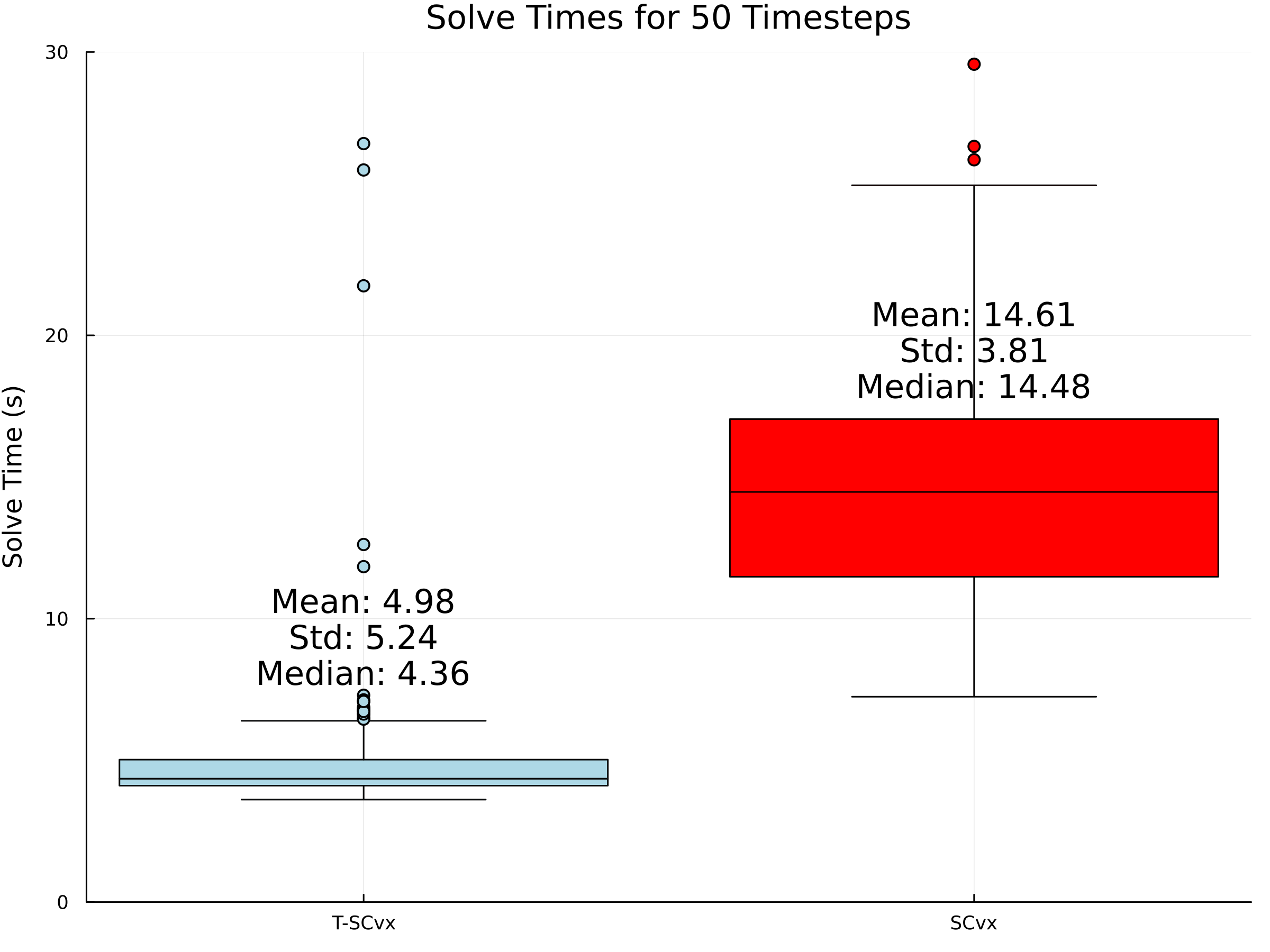}
    \caption{Box plots for~\ac{tscvx} and~\ac{scvx} showing mean, median, and standard deviations.}
    \label{fig:solveTimesSCvx}
\end{figure}

\begin{figure}
    \centering
    \includegraphics[width=0.95\textwidth]{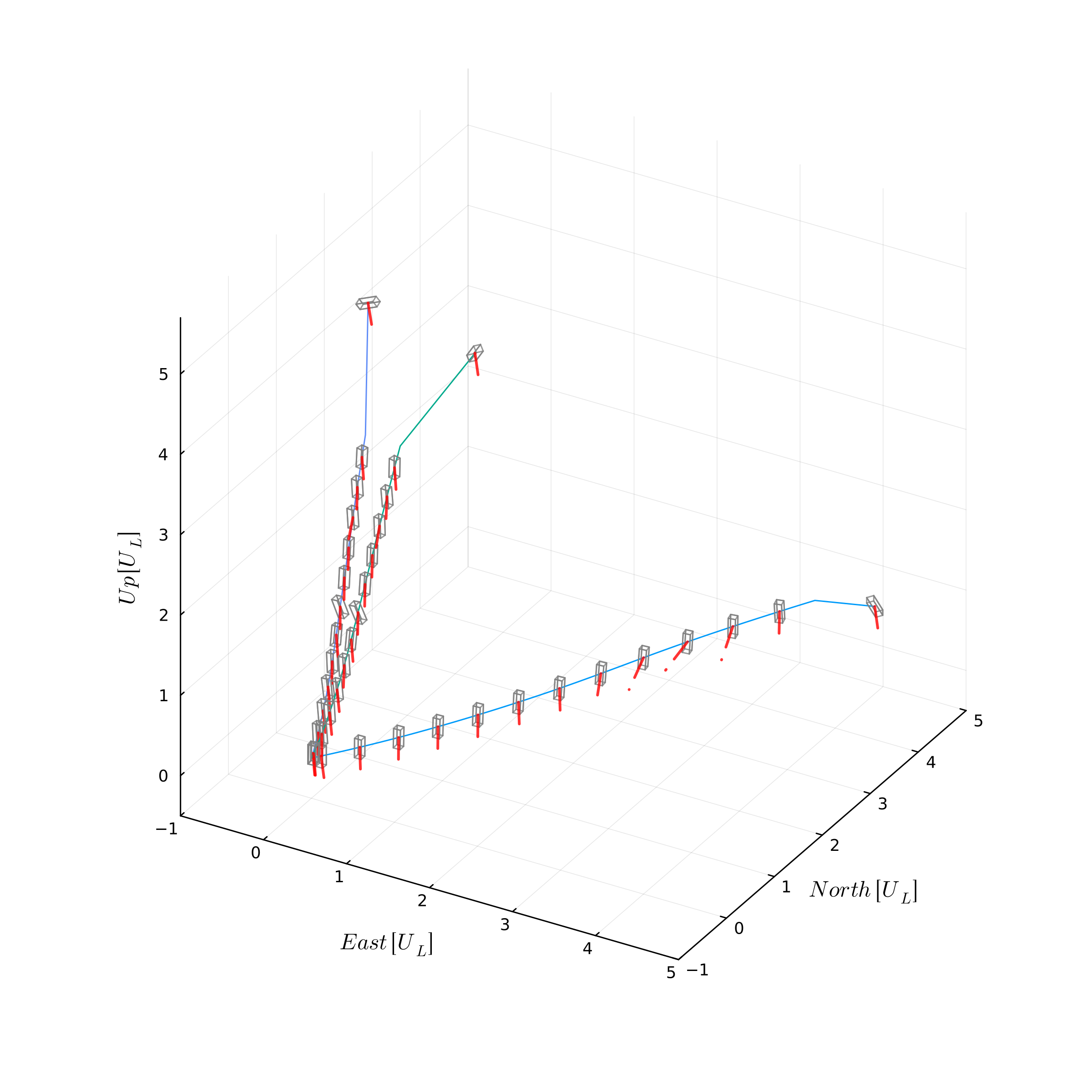}
    \caption{6-\ac{dof} minimum fuel trajectory computed by~\ac{tscvx} with thrust vectors in red.}
    \label{fig:trajectoryReduced}
\end{figure}

Figure~\ref{fig:trajectoryReduced} shows three~\ac{tscvx}-computed test trajectories.
Multiple orientations, initial positions, and velocities are resolved by the solver, computing the control strategy required to reach the final position.
Compared to~\ac{tpdg},~\ac{tscvx} extends the predictive capabilities to include not only convex tight constraints at every convex subproblem iteration but also the full state, control, and final time solution initial guess.
Returned solution feasibility is ensured by keeping the full problem penalty cost and not converging until the actual change in the penalty cost is zero.
A new addition to the trust region radius augments the contraction and growth parameters to scale by changing the number of tight constraints from the last iteration.
This update enables a reduced trust region radius when a large change in the predicted tight constraints occurs.
With the addition of tight constraint prediction and full solution prediction using transformer NNs, the mean solve time of almost 15 seconds for~\ac{scvx} was reduced to less than 5 seconds for~\ac{tscvx} when applied to the Mars 6-\ac{dof}~\ac{pdg} problem for free final time.
Given this significant decrease,~\ac{tscvx} serves as a potential candidate for high-fidelity large divert guidance trajectory generation.

\subsection{Benchmarking with Table Lookup Approaches}

In this section we demonstrate the trade-offs between lookup table methods and using~\ac{tscvx}. 
Lookup tables serve as a potential alternative to learning-based methods.
Previous work in powered descent guidance has explored linear interpolation-based lookup tables for solution and final time prediction \cite{AcikmeseBlackmoreEtAl2008, ScharfPloenEtAl2015}.
Of particular note is that these linear interpolation-based methods assume close to convex or highly discretized problems, to ensure nonlinearities in solution do not result in high prediction errors.
This section explores the effects of table lookup approaches on the non-convex 6-\ac{dof} powered descent guidance dataset.

Further, a KD-tree for nearest neighbor lookups approach was compared in this work~\cite{Bentley1975}.
For ease of comparison, the table lookup methods use the same training and test data and predict the same quantities as~\ac{tscvx}.
Due to the high inference time and memory requirements of the linear interpolation method, the training and test data were both reduced to 10 components using principal component analysis (PCA), implemented with sklearn.decomposition, and then a subset of 100 samples was used for training and 10 samples for testing~\cite{Pearson1901, PedregosaVaroquauxEtAl2011}.

Tables~\ref{tab:comparisonTableTSCvx} and~\ref{tab:comparisonTableTSCvxSolution} show the inference time, mean-squared-errors (MSE),~\ac{ood} MSEs, and peak memory usage for the KDTree and T-SCvx.
Where~\ac{ood} test data was categorized using the 95$^{th}$ percentile for the Mahalanobis distance: $d_M (\mathbf{x}, Q) = \sqrt{(\mathbf{x} - \mathbf{\mu})^T S^{-1} (\mathbf{x} - \mathbf{\mu})}$, where $Q \in \mathbb{R}^N$ is the training dataset's probability distribution with mean $\mathbf{\mu}$ and covariance $S$ and $\mathbf{x}$ is the test data point~\cite{DemaesschalckJouanrimbaudEtAl2000}.
Figure~\ref{fig:memoryvsusage} shows each algorithm type's memory usage and inference time means and standard deviations.
Results are compared against a one-second runtime requirement and 60 MB of Static Random Access Memory (SRAM)$^*$ for a RAD750 flight computer~\cite{BergerDennisEtAl2007}.

\begin{table}[h!]
    \centering
    \caption{Tight Constraints Prediction Performance}
    \begin{tabular}{lccc}
        \toprule
        \textbf{Metric} & \textbf{Linear Interpolation} & \textbf{KDTree} & \textbf{T-SCvx} \\
        \midrule
        \textbf{Inference Time (ms)} & 894.02 & \textbf{0.1500} & 4.5929 \\
        \textbf{MSE} & 0.0368 & \textbf{0.0074} & 0.0241 \\
        \textbf{\ac{ood} MSE} & 0.0700 & 0.0657 & \textbf{0.0373} \\
        \textbf{Peak Memory Usage (MB)} & 1647.0 & 3.300 & \textbf{2.634} \\
        \bottomrule
    \end{tabular}
    \label{tab:comparisonTableTSCvx}
\end{table}

\begin{table}[h!]
    \centering
    \caption{Solution Prediction Performance}
    \begin{tabular}{lccc}
        \toprule
        \textbf{Metric} & \textbf{Linear Interpolation} & \textbf{KDTree} & \textbf{T-SCvx} \\
        \midrule
        \textbf{Inference Time (ms)} & 904.43 & \textbf{0.1411} & 4.9492 \\
        \textbf{MSE} & 0.8176 & \textbf{0.6326} & 1.0450 \\
        \textbf{\ac{ood} MSE} & 2.3195 & 2.4128 & \textbf{1.1093} \\
        \textbf{Peak Memory Usage (MB)} & 1678.5 & \textbf{1.5987} & 9.0118 \\
        \bottomrule
    \end{tabular}
    \label{tab:comparisonTableTSCvxSolution}
\end{table}

\begin{figure}[hbt!]
    \centering
    \includegraphics[width=0.7\textwidth]{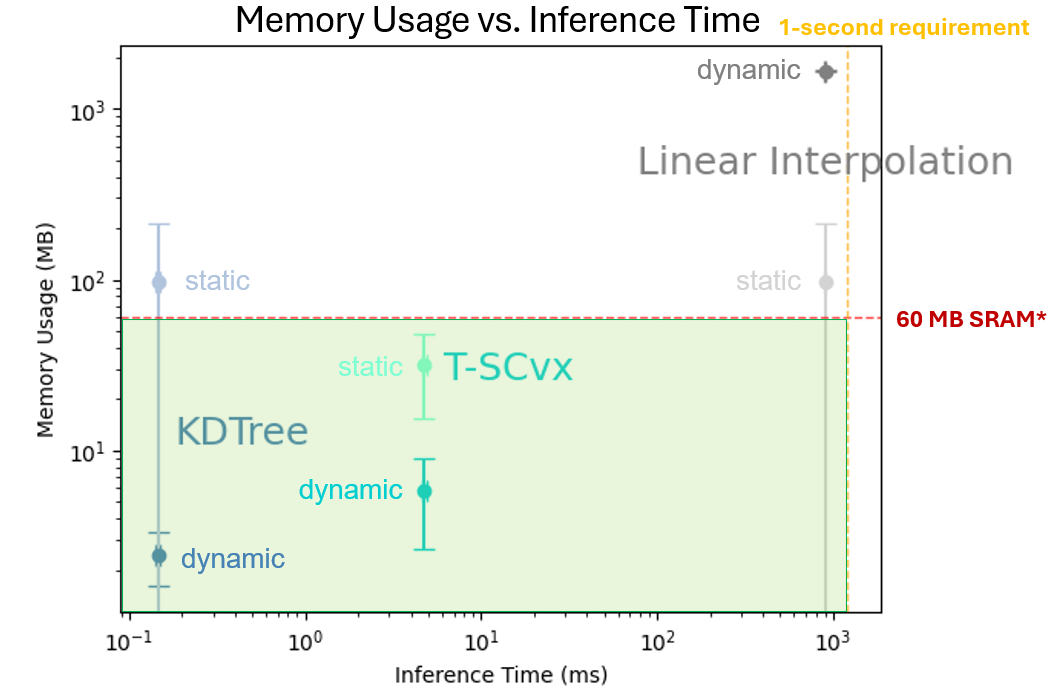}
    \caption{Memory vs. inference time for each warm-start algorithm type.}
    \label{fig:memoryvsusage}
\end{figure}

From Table~\ref{tab:comparisonTableTSCvx},~\ac{tscvx} reduces inference time, when compared to linear interpolation, by over 99\% and memory usage by 99.8\% for tight constraints prediction. 
Further, the~\ac{tscvx} has a comparatively lower MSE, when compared to linear interpolation. 
Impressively, the kd-tree nearest neighbors approach dominates all algorithms for inference time and accuracy, with sub-millisecond inference times and 0.0074 MSE.
This high accuracy does not carry over to~\ac{ood} data though, as~\ac{tscvx} dominates all methods for~\ac{ood} accuracy.
A similar pattern is observed for Table~\ref{tab:comparisonTableTSCvxSolution}, where~\ac{tscvx} dominates the linear interpolation approach for solution prediction, reducing inference time and peak memory usage by 99.5\%.
Interestingly, the linear interpolation and kd-tree lookup approaches have slightly higher accuracies, while still not maintaining the low accuracies for~\ac{ood} data.
The novel kd-tree approach is particularly performant for the solution prediction problem, dominating inference time accuracy, and peak memory usage.
It must be noted that the MSE more than doubles when~\ac{ood} samples are used for the kd-tree, as opposed to~\ac{tscvx}, which has a mostly consistent MSE when faced with~\ac{ood} data.
Overall, the kd-tree has been demonstrated to perform well for the 6-\ac{dof} tight constraint prediction and solution prediction problems, but this performance does not extend to out-of-distribution samples.
Compared to the state-of-the-art linear interpolation-based lookup table in powered descent guidance literature,~\ac{tscvx} exhibits superior performance in inference time, tight constraint prediction accuracy, and memory usage.

\section{Future Work}\label{sec: Future Directions}
While the results obtained for~\ac{tscvx} are promising for improving the computational efficiency and interpretability of the powered descent guidance problem, several future directions can be explored.
As discussed previously, knowledge of or an estimate of the set of tight constraints not only serves as a high-quality warm-start for a numerical optimization problem, but it can also provide insight into the type of convergence guarantees provided by the solver.
As the trust region formulation largely determines iteration number and convergence for the~\ac{scvx} algorithm, additional informed developments of trust region designs could be done using transformer-generated predictions.
Further, a promising area of research would be to dovetail our tight constraint prediction framework into an existing active set-based nonlinear program solver~\cite{GillMurrayEtAl2005}, potentially further increasing computational efficiency.
Finally, sparsity in linear algebra operations is advantageous for efficient linear algebra operations and future use of tight constraint predictions could inform custom solve implementations to promote sparsity in matrix operations~\cite{BookerMajumdar2021}. 

The next steps for this work include verifying the solve time gains are representative of flight-grade radiation-hardened hardware.
Novel contributions include programming, verifying, and validating the use of trained deep neural networks on hardware.
This includes testing~\ac{tpdg} and~\ac{tscvx} with custom solver implementations and developing efficient parallel algorithmic updates where possible.

\section{Conclusion}\label{sec: Conclusion}

We have demonstrated the contribution of Transformer-based Successive Convexification for 6-\ac{dof} powered descent guidance, including a 6-\ac{dof} Mars powered descent guidance test problem.
The main contributions of~\ac{tscvx} include enhanced computational efficiency and reliability in solve times for non-convex trajectory generation, enabled by the design of transformer NNs for tight constraint and solution initial guess prediction.
By mapping the initial state and constraint parameters to the tight constraints and problem state, control, and final time with a transformer NN, a variety of initial conditions and problem parameters are accounted for.
Additionally, the introduction of a rotation-invariant data augmentation method reduced the number of required samples to under 2,000 to achieve test accuracies of over 96\% and under 1 MSE.
As expected, for the 6-\ac{dof} powered descent guidance problem, Transformer-based Powered Descent Guidance demonstrated a greater reduction in mean runtime, with~\ac{tscvx} reducing the mean runtime by 66\%; runtime was reduced by 9.63 seconds when compared to~\ac{scvx}.
Benchmarking against linear interpolation approaches suggested in the literature showed~\ac{tscvx} to reduce the required inference time and memory usage by more than 99\%, with greater or comparable accuracy.
The kd-tree nearest neighbor approach serves as a more interpretable alternative, applicable when test data is not out-of-distribution.

Further, the convergence criteria remained the same as~\ac{scvx}, requiring returned trajectories to satisfy all problem constraints and maintain the feasibility of returned solutions.
An additional update from~\ac{tpdg} includes the initial mass as a parameter in~\ac{tscvx}, allowing for trajectories to be computed even after powered descent initiation.
\ac{tscvx} has been successfully demonstrated as a potential trajectory optimization for fast high-fidelity guidance scenarios, including the application problem of 6-\ac{dof} Mars landing.

\section*{Acknowledgments}
This work was supported in part by a NASA Space Technology Graduate Research Opportunity 80NSSC21K1301.
This research was carried out in part at the Jet Propulsion Laboratory, California Institute of Technology, under a contract with the National Aeronautics and Space Administration and funded through the internal Research and Technology Development program.

\bibliography{ASL,main}

\end{document}